\documentclass[reqno]{amsart}
\usepackage[english]{babel}
\usepackage[T1]{fontenc}
\usepackage[utf8]{inputenc}
\usepackage{amsmath}
\usepackage{amsfonts}
\usepackage{array}
\usepackage{eurosym}

\usepackage{amssymb}
\usepackage{mathrsfs}
\usepackage{amsthm}
\usepackage{xfrac}
\usepackage{lmodern}
\usepackage{fix-cm}
\usepackage{listings}
\usepackage{fancyvrb}
\usepackage{enumerate}   
\usepackage{xcolor}
\usepackage{pgf,tikz}
\usepackage{mathrsfs}
\usetikzlibrary{arrows}
\usepackage{float}
\usepackage{subcaption}
\usepackage{hyperref}{}
\usetikzlibrary{patterns}

\usepackage{wasysym}
\usepackage{enumitem}

\theoremstyle{plain}
\newtheorem{theorem}{Theorem}[section]
\newtheorem{lemma}[theorem]{Lemma}
\newtheorem{proposition}[theorem]{Proposition}
\newtheorem{corollary}[theorem]{Corollary}

\theoremstyle{definition}
\newtheorem{definition}[theorem]{Definition}
\newtheorem{example}[theorem]{Example}

\newtheorem{algorithm}[theorem]{Algorithm}
\theoremstyle{remark}
\newtheorem{remark}[theorem]{Remark}

\def\NZQ{\Bbb}               % the font for N,Z,Q,R,C
\def\NN{{\NZQ N}}

\def\FF{{\mathcal F}}

\def\FF{{\mathcal F}}

\def\MM{{\mathrm m}}

\def\mm{{\frak m}}
\def\nn{{\frak n}}

\newcommand{\PP}{\mathcal{P}}

\newcommand \ini{\operatorname{in}}

\newcommand{\rad}{1.5 pt}

\newcommand{\cmm}{c_G(\mm)}

\begin{document}
		\title{Diagonal F-threshold of binomial edge ideals}
		\author[Giancarlo Rinaldo]{Giancarlo Rinaldo}
\email{giancarlo.rinaldo@unime.it}
\address{Department	of Mathematics, Informatics, Physics and Earth Science, University of Messina, Viale F. Stagno d’Alcontres, 31, Messina, 98166, Italy}
	\author[Francesco Romeo]{Francesco Romeo}
	\email{francesco.romeo@unicas.it}
	\address{Department of Electrical and Information Engineering, University of Cassino and Southern Lazio, Via G. Di Biasio 43, 03043 Cassino, Italy.}

\begin{abstract}
	We compute the diagonal $F$-threshold of binomial edge ideals associated to graphs in the case of cycles, cycles with whiskers, bipartite graphs, and block graphs. Moreover, we provide general bounds for any graph.
\end{abstract}

\maketitle
		
\section{Introduction}
The theory of $F$-thresholds provides a fundamental tool for studying singularities in rings of positive characteristic. Introduced by Mustaţă, Takagi and Watanabe in the setting of regular local rings \cite{Musta2004FthresholdsAB}, the notion of $F$-threshold was conceived as a characteristic $p$ counterpart of the log-canonical threshold \cite{TAKAGI2004278}. Later, Huneke, Mustaţă, Takagi and Watanabe generalized this invariant and established connections with tight closure theory, integral closure and Hilbert--Samuel multiplicities \cite{HMTW2008}.

Several explicit computations of $F$-thresholds have subsequently appeared in the literature. Matsuda, Ohtani and Yoshida determined formulas for binomial hypersurfaces in \cite{doi:10.1080/00927870903107940}, while Chiba and Matsuda investigated the case of Hibi rings, namely graded toric rings arising from finite distributive lattices, in \cite{CM2015}. Moreover, De Stefani, Nuñez-Betancourt and Pérez proved the existence of $F$-thresholds for arbitrary Noetherian rings in \cite{destefani2017existence}. Despite these developments, obtaining explicit values of $F$-thresholds remains difficult in many relevant classes of ideals and rings.

The case of determinantal ideals has been studied in \cite{BMRS} and \cite{BRRS}. In this work, our attention is devoted  to the diagonal $F$-threshold associated with binomial edge ideals. Binomial edge ideals associated to simple graphs have been intensively studied in the last decade (see \cite{HHO-book}, Chapter 7). Their algebraic and homological properties are intimately related to the combinatorics of the underlying graph. Let $G$ be a simple graph (that is, undirected, with no loops, and no multiple edges) on
the vertex set $[n] := \{1, 2, \ldots , n\}$ and $S = K[x_1,\ldots, x_n , y_1 ,\ldots,y_n]$ the polynomial ring in $2n$ variables. The binomial edge ideal $J_G\subset S$ of $G$ is generated by all the binomials of
the form $f_{ij} = x_i y_j - x_j y_i$ where $\{i, j\}$ is an edge of $G$. In other words, $J_G$ is generated by the 2-minors of the generic matrix 
\[
M=\begin{pmatrix}
	x_1 &\ldots & x_n\\
	y_1 &\ldots & y_n
\end{pmatrix},
\]	
which correspond to the edges of $G$ (see \cite{HHHKR}, \cite{MO}, and \cite{HHO-book}). We recall the definition of $F$-threshold. Let $R$ be a ring of characteristic $p > 0$ and $I, J \subseteq R$ two
ideals such that $I \subseteq \sqrt{J}$. For a fixed positive integer $e$ we define the finite integer
\[
v_I^J(p^e):=\max\{r \in \NN | I^r\not\subseteq J^{[p^e]}\}.
\]
The $F$-threshold of $I$ with respect to $J$ is the limit
\[
c^J(I) : = \lim_{e\to \infty}\frac{v_I^J(p^e)}{p^e}
\]
Following the arguments of \cite{BMRS}, we can consider any positive integer $k$ as the argument of the $v$ function.
In our setting, if $G$ is a graph then for any $k \in \NN$ we call 
\[
v_G(k):=\max\{r \in \NN \ | \ \mm^r\not\subseteq (\mm^{[k]},J_G)\},
\]
and
\[
c_G(\mm) :=  \lim_{k\to \infty}\frac{v_G(k)}{k},
\]
namely the diagonal $F$-threshold $c^\mm(\mm)$ of $S/J_G$.
It is worth mentioning that in the article \cite{LaClair2025}, the author studies $c^{\mm}(J_G)$ relating this invariant to linear programs. 
In this paper we study $\cmm$ and our focus is to find strong relations between combinatorics of $G$ and $\cmm$.

One of the first observations is that
\begin{equation}\label{eq:n}
\cmm\geq n,
\end{equation}
in particular cycles have $\cmm =n$ (see Section \ref{sec:pre}). Then, we focus on independent sets of vertices and matchings. In particular, if a graph has a perfect matching, then we prove that $\cmm=n$; moreover from K\"onig theorem, we derive that a bipartite graph has $\cmm=2d$ where $d$ is the cardinality of maximum independent set (see Section \ref{sec:indmat}). This result inspired us to extend the concept of matching  
by considering cliques of higher cardinality instead of edges (see Section \ref{sec:bounds}).
A graph $G$ is said \emph{clique-coverable} if there exists a partition of the vertices $V(G)$ in cliques of $G$. We generalize this notion considering clique-coverings of subgraphs of $G$. Let $H$ be a subgraph of $G$ that is clique-coverable and let $P_H$ be a partition of $H$ in cliques. We refer to $P_H$ as a \emph{clique matching}.
We call $b(G)=\max\{|V(H)| : \ P_H \mbox{ clique matching of } G \}$.
Thanks to this combinatorial invariant we obtain the following upper  bound
\begin{equation}\label{eq:2nbg}
	\cmm\leq 2n-b(G).
\end{equation}
We also prove that the equality holds in the case of block graphs  (see Section \ref{sec:blocks}), but it does not in the case of odd cycles. Hence in general we obtain that
\begin{equation}\label{eq:bounds}
	n\leq \cmm\leq 2n-b(G).
\end{equation}
and there are many classes of graphs for which one of the bounds is an equality. In fact, $\cmm=n$ when particular subgraphs cover all the vertices, as in the cases of clique-coverable graphs and graphs with a $2$-factor. It is of interest to find good subgraphs $H$ that induce $\cmm$ attaining one of the two bounds.

\section{Tensor product and cycles}\label{sec:pre}
In this section, we recall properties of the diagonal $F$-threshold and we compute $\cmm$ for cycles.

Let $R=K[x_1, \ldots, x_n]$, $S=K[y_1, \ldots, y_m]$, $T=K[x_1, \ldots, x_n,y_1, \ldots, y_m]$, $\mathfrak{n}=(x_1,\ldots,x_n), \mathfrak{m}=(y_1,\ldots, y_m)$, let $I \subseteq R$ and $J \subseteq S$ be ideals, let $M=R/I$, $N=S/J$ and $L=T/(I+J)$. It is clear that 
\[
L=M\otimes N.
\]
Moreover we have the following
\begin{lemma}\label{lem:tp}
In the previous setting, 
\[
c^{\mathfrak{m} + \mathfrak{n}}(\mathfrak{m} + \mathfrak{n})=c^{\mm}(\mathfrak{m})+c^{\mathfrak{n}}(\mathfrak{n}).
\]
\end{lemma}
\begin{proof}
Let $u$ be a  monomial of degree $v_\mm^\mm(p^e)$ in $R$ such that $u \notin (I,\mm^{[p^e]})$, and let $u'$ be a  monomial of degree $v_\nn^\nn(p^e)$ in $S$ such that $u' \notin (J,\nn^{[p^e]})$, then $uu' \notin  (I+J,(\mm+\nn)^{[p^e]})$. Conversely, any monomial of degree $v_\mm^\mm(p^e)+v_\nn^\nn(p^e)+1$ lies in $(I+J,(\mm+\nn)^{[p^e]})$ because its degree is either greater than $v_\mm^\mm(p^e)$ in $x_1,\ldots,x_n$ or greater than $v_\nn^\nn(p^e)$ in $y_1,\ldots,y_m$.
\end{proof}

Let $S=K[x_1,\ldots,x_n,y_1,\ldots,y_n]$ be the polynomial ring in $2n$ variables. We represent any monomial in $S$ with a $2 \times n$ matrix representing the multidegree in the following way: if $u=x_1^{\alpha_1}\cdots x_n^{\alpha_n} y_1^{\beta_1}\cdots y_n^{\beta_n}$, then
\begin{equation}\label{eq: M(u)}
	M(u)=\begin{pmatrix}
		\alpha_1 &\ldots & \alpha_n\\
		\beta_1 &\ldots & \beta_n
	\end{pmatrix},
\end{equation}
and we write $u=X^{M(u)}$.
Let $G$ be a finite simple graph and  let $J_G$ be its binomial edge ideal. To compute $\cmm$, we study the containment 
\[
\mm^r \subset (\mm^{[k]},J_G)
\]
for any $k$ in $\NN$ and suitable $r \in \NN$.
For this aim, we observe that given a binomial $f_{ij}=x_iy_j-x_jy_i$ and a monomial $u \in S$ with matrix $M(u)$ as in Equation \eqref{eq: M(u)},  the substitution $x_iy_j=x_jy_i$ produces a monomial $u'$ such that 
\[
M(u')=\begin{pmatrix}
	\ldots &\alpha_i-1 &\ldots & \alpha_{j}+1 &\ldots\\
	\ldots &\beta_i+1 &\ldots & \beta_j-1 &\ldots
\end{pmatrix}.
\]

\begin{lemma} \label{ex:2x2 reduction}
	Let $k\in \NN$, let $G$ be a graph and let $\{i,j\}\in E(G)$. 
	Let $u$ be a monomial with matrix $M(u)$. If $\alpha_i+\alpha_j+\beta_i+\beta_j \geq 2(k-1)+1$, then $u \in (\mm^{[k]},J_G)$.
\end{lemma}
\begin{proof}
	We consider the following submatrix of $M(u)$
	\[
	\begin{pmatrix}
		\alpha_i & \alpha_j\\
		\beta_i &\beta_j
	\end{pmatrix}
	\]
	Without loss of generality we assume $\alpha_i\leq\beta_j$, after $\alpha_i$ steps of reductions induced by the relation $x_iy_j=x_jy_i$, we obtain the submatrix
	\[
	\begin{pmatrix}
		0 & \alpha_j+\alpha_i\\
		\beta_i+\alpha_i &\beta_j-\alpha_i
	\end{pmatrix} = \begin{pmatrix}
		0 & \alpha_j'\\
		\beta_i' &\beta_j'
	\end{pmatrix}.
	\]
	Now, we assume $\alpha_j'\leq \beta_i'$. If $\alpha_j'+\beta_j'\leq k-1$ then $\beta_i'\geq k$ and $u \in (\mm^{[k]},J_G)$. If $\alpha_j'+\beta_j'\geq k$, after $\alpha_j'$ reductions we have
	\[
	\begin{pmatrix}
		\alpha_j' & 0\\
		\beta_i' -\alpha_j' &\beta_j'+\alpha_j'
	\end{pmatrix},
	\]
	and $\beta_j'+\alpha_j'\geq k$, and $u \in (\mm^{[k]},J_G)$, too.
\end{proof}	

We recall the notion of admissible path, introduced in  \cite{HHHKR} in order to compute Gr\"obner bases of binomial edge ideals. A path $\pi:i=i_0,i_1,\ldots,i_r=j$ in a graph $G$ is called {\em admissible}, if
\begin{enumerate}
	\item $i_k\neq i_\ell$ for $k\neq \ell$;
	\item for each $k=1,\ldots,r-1$ one has either $i_k<i$ or $i_k>j$;
	\item for any proper subset $\{j_1,\ldots,j_s\}$ of $\{i_1,\ldots,i_{r-1}\}$, the sequence $i,j_1,\ldots,j_s,j$ is not a path.
\end{enumerate}
Given an admissible path $\pi:i=i_0,i_1,\ldots,i_r=j$ from $i$ to $j$ with $i<j$ we associate the monomial  $u_\pi=(\prod_{i_k>j}x_{i_k})(\prod_{i_\ell<i}y_{i_\ell})$. In \cite{HHHKR} it is shown that
\[
\ini_< (J_G)=(x_iy_ju_\pi\:\; \pi \mbox{ is an admissible path}).
\]

Let $u=z_{i_1}^{a_{i_1}}\cdots z_{i_r}^{a_{i_r}}$ be a monomial in $S$, with $z_{i_j}\in \{x_{i_j},y_{i_j}\}$ and $a_{i_j}>0$. We set $V(u)=\{i_1,\ldots,i_r\}$. 
\begin{proposition}\label{pro:groeb}Let $G$ be a graph on $[n]$. Then
	the reduced Gr\"obner basis of $(\mm^{[k]},J_G)$ is $\mathcal{M}\cup\mathcal{B}$, where $\mathcal{M}$ is a set of monomials, and $\mathcal{B}$ is the set of binomials which coincides with the  Gr\"obner basis of $J_G$.
	Moreover, for all $u\in \mathcal{M}$ we have $\deg(u)\geq k$ and
	\begin{enumerate}
		\item if $\deg(u)=k$ then $u\in \mm^{[k]}$;
		\item if $\deg(u)>k$ then $G_{V(u)}$ is a connected graph and there exists $x_i\mid u$ and $y_j\mid u$.
	\end{enumerate} 
\end{proposition}
\begin{proof} The first statement holds, observing that the only binomials in the ideal $(\mm^{[k]},J_G)$ are the ones in $J_G$ and the $S$-pair of a monomial with a binomial is a monomial itself. Hence the only binomials in the Gr\"obner basis are the ones of the Gr\"obner basis of $J_G$. To prove the second statement, since the $S$-pair of two monomials is $0$, we focus on the S-pair
	\begin{equation}\label{eq:monspair}
		u'=S(u,u_\pi f_{ij}) \neq 0,
	\end{equation}
	for some monomial $u$ and $u_\pi f_{ij}\in \mathcal{B}$.
	We assume $u'$ in \eqref{eq:monspair} reduced by the division algorithm. In fact if $u'$ is not reduced then there are binomials in $\mathcal{B}$ that reduce it either to zero or to a new monomial $u'' \neq 0$ such that $V(u')=V(u'')$. Namely, these reductions are invariant with respect to the function $V(\_)$. 
	
	We first show that if $G_{V(u)}$ is connected then either $G_{V(u')}$ is connected or $u'$ is $0$. At first we assume $u\in\{x_1^k,\ldots,x_n^k,y_1^k,\ldots,y_n^k\}$. We have that if $u$ and $u_\pi x_iy_j$ are coprime then $u'$ is $0$. If $u$  and $u_\pi x_i y_j$ are not coprime  then $V(u)\subset V(u_\pi x_i y_j)$.  Hence $S(z_h^k, u_\pi f_{ij})=u'$ with $z_h\in\{x_h,y_h\}$ for some $h \in \{1,\ldots, n\}$ where $V(u')=V(u_\pi x_i y_j)$, that is $G_{V(u')}$ is a path, hence connected. By the same argument, if $u$ is such that $G_{V(u)}$ is connected in \eqref{eq:monspair} and $u$ is not coprime with $u_\pi x_i y_j$, then $V(u')=V(u)\cup V(u_\pi x_i y_j)$. Namely, $V(u')$ is the union of the set of the vertices of the path $\pi$  and the set of vertices of the connected graph $G_{V(u)}$, and they have at least a common vertex since $u$ and $u_\pi x_iy_j$ are not coprime. Hence, $G_{V(u')}$ is connected.
	Now, focus on the degree of $u'$. Suppose that $u'\in \mathcal{M}$ with $u'\notin \mm^{[k]}$. Then $u'$ is as in \eqref{eq:monspair}, and it is a multiple of the trailing monomial of the homogeneous binomial $u_\pi f_{ij}$, that is $u_\pi y_jx_i$ and the statement follows  easily.
\end{proof}

\begin{lemma}\label{lem:geq}
Let $G$ be a graph. Then 
\[
\cmm \geq n
\]
\end{lemma}
\begin{proof}
To prove the claim we prove that 
	\[
	v_G(k)\geq n (k-1),
	\]
	that is we find a monomial of degree $n(k-1)$ that is not in $(\mm^{[k]},J_G)$, and we conclude by taking the limit 
	\[
	\cmm= \lim_{k\to \infty} \frac{v_G(k)}{k} \geq \lim_{k\to \infty} \frac{n(k-1)}{k}.
	\]
	By Proposition \ref{pro:groeb} we have 
	\[
	u=\prod_{i=1}^n x_i^{k-1} \notin (\mm^{[k]},J_G),
	\]
	hence the assertion follows.
\end{proof}
\begin{corollary}\label{cor:K2}
Let $G=K_2$, then $\cmm=2$.
\end{corollary}
\begin{proof}
From Lemma \ref{lem:geq}, we have $\cmm\geq 2$. 
For any $k \in \NN$, we prove $v_G(k)\leq 2 (k-1)$, by proving that any monomial $u$ of degree $2(k-1)+1$ is in $(\mm^{[k]},J_G)$.  We observe that we have 
\[
M(u)=	\begin{pmatrix}
		\alpha_1 & \alpha_2\\
		\beta_1 &\beta_2
	\end{pmatrix}
\]
and by applying Lemma \ref{ex:2x2 reduction} the assertion follows.
\end{proof}

\begin{theorem}\label{the:oddcycles}
	Let $G$ be a cycle on $n$ vertices. Then
	\[
	\cmm= n
	\]
\end{theorem}
\begin{proof}
	We observe that the case of $n$ even is proved in a more general setting in Theorem \ref{the:bipartite}, observing that a even cycle has exactly $n/2$ independent vertices. So we focus on the case $n$ odd.
	
	To prove the claim we prove that 
	\[
	v_G(k)=n (k-1),
	\]
	that is any monomial in $\mm^{v_G(k)+1}$ is in $(\mm^{[k]},J_G)$.
     The inequality $v_G(k)\geq n (k-1)$ follows from Lemma \ref{lem:geq}.
	 To prove $v_G(k)\leq n (k-1)$,  we  consider a generic monomial $u$ of degree 
	\[
	n(k-1)+1,
	\]
	with matrix $M(u)$, and show after reduction, that we obtain a new representation of $u$, namely $u'$, such that $M(u')$ has an entry with value greater than or equal to $k$, that is $u\cong u'\in  (\mm^{[k]},J_G)$.
	Moreover, we assume that $E(G)=\{\{1,2\},\{2,3\},\ldots,\{n-1,n\}\}\cup \{\{1,n\}\}$.
	Let $M_{n-1}$ be the submatrix of $M(u)$ on the first $n-1$ columns, namely there is only one column, the last one $C_n$, that is removed. 	\[
	M_{n-1}=\begin{pmatrix}
		\alpha_1 & \alpha_2 &\ldots & \alpha_{n-2}& \alpha_{n-1}\\
		\beta_1 & \beta_2& \ldots & \beta_{n-2}& \beta_{n-1}
	\end{pmatrix}, C_n= \begin{pmatrix}
		\alpha_n\\
		\beta_n
	\end{pmatrix}.
	\]
	Since $n-1$ is even, we can partition $M_{n-1}$ into $(n-1)/2$ $2\times 2$ submatrices. By Lemma \ref{ex:2x2 reduction}, if only one of this $2\times 2$ submatrices has sum greater than to $2(k-1)$ then $u$ belongs to 
	$(\mm^{[k]},J_G)$. So each submatrix has sum at most $2(k-1)$. That is all submatrix $M_{n-1}$ has sum at most 
	\[
	\frac{n-1}{2}\cdot 2(k-1)=(n-1)(k-1).
	\]
	This implies that $C_n$ has sum at least $k$. By the same argument, if we isolate the first column, $C_1$,  and consider as a submatrix $M_{n-1}$ the one containing all the column but the first one, we have that the column $C_1$ has sum at least $k$. But in this case, since $\{1,n\}$ is an edge we have the submatrix containing the columns $C_1$, and $C_n$, namely 
	\[
	\begin{pmatrix}
		\alpha_1 & \alpha_n\\
		\beta_1 &\beta_n
	\end{pmatrix}
	\]
	has sum greater than $2k$, and by Lemma \ref*{ex:2x2 reduction}, $u\in (\mm^{[k]},J_G)$.
\end{proof}

\section{Independent sets and matchings}\label{sec:indmat}
In this Section we prove the strict connection between independent sets of vertices and matchings with $\cmm$. In particular, we compute it when $G$ is a bipartite graph.
\begin{lemma}\label{lem:independentset}
	Let $G$ be a graph on $[n]$. Let $A$ be a maximal independent set of $G$. Then
	$\cmm\geq 2 |A|$.
\end{lemma}	
\begin{proof}
	Let $A=\{i_1,\ldots, i_a\}$, and  $u=(x_{i_1}y_{i_1}x_{i_2}y_{i_2}\cdots x_{i_a}y_{i_a})^{k-1}$. We prove that $u\notin (\mm^{[k]},J_G)$. Obviously $u\notin \mm^{[k]}$. Moreover, by Proposition \ref{pro:groeb} $u\notin (\mm^{[k]},J_G)$, in fact  $u$ to be  multiple of a monomial in  $(\mm^{[k]},J_G)\setminus (\mm^{[k]})$ must have an edge in $G_{V(u)}$, that is impossible since it is a graph of isolated vertices.
\end{proof}

\begin{lemma}\label{lem:subgraph}
	Let $G$, $H$  be two graphs with $n$ and $n'$ vertices with $H\subseteq G$. Let $\mm\subset K[\{x_i,y_i\}_{i\in [n]}]$, and $\mm'\subset K[\{x_i,y_i\}_{i\in [n']}]$ be the respective maximal ideals. Then  
	\[
	\cmm\leq c_{H}(\mm')+2(n-n').
	\]
\end{lemma}	
\begin{proof}
We start considering $J_H\subseteq K[\{x_i,y_i\}_{i\in [n]}]$.	
We observe that if exists $r\in \NN$ with 
\[
\mm^r\subseteq (\mm^{[k]},J_H)
\]
then $\mm^r\subseteq (\mm^{[k]},J_G)$. In fact it is straightforward the following containment $(\mm^{[k]},J_H)\subseteq (\mm^{[k]},J_G)$. Hence, 
\[
\max\{r \in \mm^r\not\subseteq (\mm^{[k]},J_G)\}\leq \max\{r \in \mm^r\not\subseteq (\mm^{[k]},J_H)\}.
\]
Moreover, since 
\[
K[\{x_i,y_i\}_{i\in [n]}]/J_H\cong K[\{x_i,y_i\}_{i\in [n']}]/J_H\otimes K[\{x_i,y_i\}_{i\in [n] \setminus [n']}],
\]
then the assertion follows by Lemma \ref{lem:tp}. 
\end{proof}
\begin{proposition}\label{prop:pm}
	Let $G$ be a graph on $[n]$ vertices with a perfect matching. Then
	\[
	\cmm=n.
	\]
\end{proposition}
\begin{proof}
	From Lemma \ref{lem:geq} we have that $\cmm \geq n$. We prove $\cmm \leq n$. 
	Let $H$ be the subgraph on the perfect matching. Since $V(H)=V(G)$, then from Lemma \ref{lem:subgraph} we have that
	\[
	\cmm\leq c_{H}(\mm).
	\]
	Moreover, $n=2m$ and the subgraph $H$ is a disjoint union of $m$ $K_2$ graphs, hence from Lemma \ref{lem:tp}, we have 
	\[
	c_{H}(\mm)=n
	\]
	and the assertion follows.
\end{proof}

\begin{corollary}\label{cor:matchingbound}
	Let $G$ be a graph with matching number $\MM(G)$. Then 
	\[
	\cmm\leq 2(n-\MM(G)).
	\]
\end{corollary}
\begin{proof}
	Let $H$ be a maximal matching of $2\MM(G)$ vertices. Since $H$ is a perfect matching on itself, by Proposition \ref{prop:pm} and Lemma \ref{lem:subgraph}, then we have respectively  $c_{H}(\mm')=2\MM(G)$, and $$\cmm\leq 2\MM(G)+2(n-2\MM(G))=2(n-\MM(G)).$$
	
\end{proof}

\begin{theorem}\label{the:bipartite}
	Let $G$ be a bipartite graph.  Then
	\[
	\cmm=2d
	\]
	where $d$ is the cardinality of a maximum independent set.
\end{theorem}
\begin{proof}
	By Lemma \ref{lem:independentset}, we have $\cmm \geq 2d$. 
	For the other inequality, from K\"oning Theorem we have that $d=n-\mathrm{m}(G)$, hence by applying Corollary \ref{cor:matchingbound} the assertion follows. 
\end{proof}

\begin{remark}\label{rem:nearlyperfectmatchings}
	By Theorem \ref{the:oddcycles}  and Theorem \ref{the:bipartite} cycles have always $\cmm=n$ with $n$ even or odd. For paths instead the $\cmm$ is $n$ if $n$ is even, and it is $n+1$ if $n$ is odd. In fact, a path on $2k+1$ vertices has a maximal independent set of $k+1$ vertices. 
\end{remark}

\section{General bounds for $\cmm$}\label{sec:bounds}
We start this section with the following
\begin{definition}\label{def:CI}
We define a $CI$-matching $P$ a pair $(A,B)$ where $A$ is an independent set of vertices $G$ and $B$ is a set of disjoint cliques of $G$ such that if $i\in A$ and $j\in V(B)$ then $\{i,j\}\notin E(G)$, where $V(B) := \bigcup\limits_{C \in B} V(C) $.
Given a $CI$-matching $P$ we call $a_P=|A|$ and $b_P=|V(B)|$ and we set $V(P)=A \cup V(B)$. When $A= \emptyset$ (resp. $B=\emptyset$), we simply write $P=B$ (resp. $P=A$).
\end{definition}
\begin{proposition}\label{pro:lowupbound}
	Let $G$ be a graph on $n$ vertices, $P=(A,B)$ be a $CI$-matching. Then
	\begin{equation}\label{eq:ci-bound}
		2a_P+b_P\leq \cmm\leq 2n-b_P 
	\end{equation}
\end{proposition}
\begin{proof}
	Let $u$ be the following monomial
	\[
	u=\prod_{i\in A} x_iy_i\prod_{j\in V(B)}x_j,  
	\]
	then $u^{k-1}\notin (\mm^{[k]},J_G)$ from Proposition \ref{pro:groeb}. Hence, 
	\[
	\cmm\geq 2a_P+b_P.
	\]
	Let $H$ be the subgraph of $G$ such that 
	\[
	V(H)=A\cup V(B),\text{ and   }E(H)=E(B).
	\] 
	Then by Lemma \ref{lem:subgraph} we have
	\[
	c_G(\mm)\leq c_H(\mm')+2(n-(a_P+b_P)). 
	\]
	From Lemma \ref{lem:tp} we have $c_H(\mm')=2a_P+b_P$. That is
	$c_G(\mm)\leq 2n-b_P$.
\end{proof}
\begin{remark} We observe that the lower bound of \eqref{eq:ci-bound} can be improved in general taking as value $\max\{n,2a_P+b_P\}$, where $n$ is induced by the monomial $u=\prod_{i\in V(G)}x_i^{k-1}$. 
\end{remark}

Given two different CI-matchings $P$ and $P'$ on the same graph $G$, one has 
	\[
	\max\{2a_P+b_P,2a_{P'}+b_{P'}\} \leq \cmm \leq 2n- \max\{b_P,b_{P'}\},
	\]
hence the following definition naturally arises 
\begin{definition}
Let $P$ be a CI-matching on a graph $G$.
\begin{enumerate}
\item We say that $P$ is \textit{Left-maximal (L-maximal)} if $2a_{P}+b_P \geq 2a_{P'}+b_{P'}$ for any CI-matching $P'$ of $G$.
\item We say that $P$ is \textit{Right-maximal (R-maximal)} if $2n-b_P \leq 2n-b_{P'}$ (equivalently $b_P \geq b_{P'}$) for any CI-matching $P'$ of $G$.
\end{enumerate} 
\end{definition}

\begin{remark}
Let $G$ be a graph and let $Q$ be a maximal clique matching of $G$ such that $|V(Q)|=b(G)$. Then $Q$ is a R-maximal CI-matching of $G$.
\end{remark}
We observe that there are CI-matchings $P,P'$ such that $2a_P+b_P=2a_{P'}+b_{P'}$ as in the following
\begin{example}\label{exa:equi}
Take a block star graph $G$ that has only one whisker, e.g.
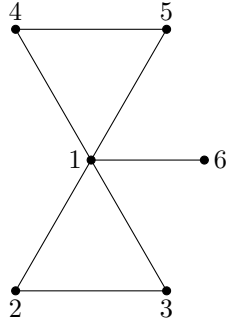
\begin{figure}[H]
\begin{tikzpicture}
\filldraw (1,1.73) circle (\rad) node [anchor=east]{1};
\filldraw (0,0) circle (\rad) node [anchor=north]{2};
\filldraw (2,0) circle (\rad) node [anchor=north]{3};
\filldraw (0,3.46) circle (\rad) node [anchor=south]{4};
\filldraw (2,3.46) circle (\rad) node [anchor=south]{5};
\filldraw (2.5,1.73) circle (\rad) node [anchor=west]{6};
\draw (0,0)--(2,0)--(1,1.73)--cycle;
\draw (0,3.46)--(2,3.46)--(1,1.73)--cycle;
\draw (2.5,1.73)--(1,1.73);
\end{tikzpicture}
\end{figure}
Then, $P=(\{6\},\{\{2,3\},\{4,5\}\})$ and $P'=\{\{2,3\},\{4,5\},\{1,6\}\}$ are CI-matchings of $G$ and $6=2a_P+b_P=2\cdot 1 +4 =2\cdot 0+ 6=2a_{P'}+b_{P'}$.
\end{example}

\begin{definition}
Two CI-matchings $P,P'$ are said \emph{equivalent} if $2a_P+b_P=2a_{P'}+b_{P'}$. Moreover, a CI-matching $P$ is said \emph{reduced} if there is no equivalent CI-matching $P'$ with $b_{P'} > b_P$.
\end{definition}
In particular, the CI-matching $P$ in Example \ref{exa:equi} is equivalent to the CI-matching $P'$ that is reduced.
\begin{lemma}\label{lem:notred}
Let $G$ be a connected graph and let $P$ be a $CI$-matching of $G$ with $a_P >0$. If there exists $v \in V(G)\setminus V(P)$ such that $|N(v)\cap A_P|=1$, then $P$ is not reduced.
\end{lemma}
\begin{proof}
Let $|N(v)\cap A_P|=\{w\}$. If
\[
P'= (A_P\setminus \{w\},B_{P}\cup \{\{v,w\}\}),
\]
then $P'$ is a CI-matching with $a_{P'}=a_P-1$ and $b_{P'}=b_{P}+2$. We have $2a_{P'}+b_{P'}=2a_P+b_P$ hence $P$ and $P'$ are equivalent with $b_{P'}>b_{P}$, hence $P$ is not reduced. 
\end{proof}

Lemma \ref{lem:notred} offers a way to construct a reduced CI-matching $P$.
\begin{algorithm} \ \\\textbf{Input:} a CI-matching $P$\\
\textbf{Output:} a reduced CI-matching \\
\begin{enumerate}
\item Let $F=\{v \in V(G)\setminus V(P): |N(v)\cap A_P|=1 \}$. If $F = \varnothing$, then return $P$. Else $F=\{v_1,\ldots,v_\ell\}$ for some $v_i \in V(G)$.
\item For $i=1,\ldots,\ell$, let $w_i$ be such that $N(v_i)\cap A_P=\{w_i\}$. Set
\[
P'=(A_{P}\setminus \{w_1,\ldots, w_\ell\}) \cup (B_P \cup  \{\{v_i,w_i\}_{i=1,\ldots, \ell}\}) 
\]
\item Set $P=P'$ and return to Step (1).
\end{enumerate}
\end{algorithm}
The algorithm terminates because $|V(G)\setminus V(P')|< |V(G) \setminus V(P)|$ and $a_{P'} < a_{P}$. For the aim of describing the vertices in an L-maximal CI-matching, we consider the following 
\begin{lemma}
%Let $G$ be a connected graph and let $P$ be an $L$-maximal matching of $G$. If $G$ has a whisker $\{v,w\}$ with $w$ free vertex, then $w\in V(P)$.

Let $G$ be a connected graph and let $P$ be an $L$-maximal matching of $G$. If $G$ has a whisker $\{v,w\}$ with $w$ free vertex, then there exist an $L$-maximal matching $P'$ with $w\in V(P')$.

\end{lemma}
\begin{proof}
%By contraposition, assume that $w \in V(G) \setminus V(P)$. It follows that $v \in V(P)$, otherwise we can add the edge $\{v,w\}$ to $P$. \\

We assume that $w \in V(G) \setminus V(P)$. It follows that $v \in V(P)$, otherwise we can add the edge $\{v,w\}$ to $P$, namely $P$ is not $L$-maximal. \\
If $v \in A_P$, we can replace $v$ with $w$ obtaining the thesis.\\
%If $v \in V(B_P)$ and let $K_m$ be the clique containing $v$. If $m  > 2$, then we consider $K_m\setminus \{v\}$ and $\{v,w\}$ obtaining a bigger CI-matching.\\
If $v \in V(B_P)$, let $K_m$ be the clique containing $v$. If $m  > 2$, then we consider 
\begin{equation}\label{eq:m>2}
P'=(A_P,B_P\setminus\{K_m\} \cup\{K_m\setminus \{v\},\{v,w\}\}).
\end{equation}
We observe that $P$ is not $L$-maximal since $a_P=a_{P'}$ and $b_{P'}=b_P+1$. If $m=2$, we construct a new $L$-maximal partition $P'$ with $w\in V(P')$. Obviously $P'$ does not contain the edge $K_2=\{v,u_1\}$ of $B_P$. We focus on the vertex $u_1$. If $u_1$ can be also added to $A_{P'}$, then we obtain a CI-matching with a larger $A_{P'}$. Namely,
\[
P'=(A_P \cup \{w,u_1\},B_P\setminus\{K_2\})
\] 
In this case $a_{P'}=a_P+2$ and $b_P=b_{P'}-2$, and we have that $P$ is not $L$-maximal, since $2a_{P'} + b_{P'}=2a_P + b_P+2$. Hence, for the $L$-maximality of $P$ and from Definition \ref{def:CI}, $u_1$ is adjacent to a vertex $u_2 \in B_P$. Here we can apply the same consideration for the vertex $v$. Namely, there is a $K_m\in P$ with $m\geq 2$ with $u_2\in K_m$. If $m>2$ we use a similar argument to the one related with equation \eqref{eq:m>2}. That is we define 
\begin{equation}\label{eq:m>2}
	P'=(A_P,B_P\setminus\{K_m\}\cup\{K_m\setminus \{u_2\},\{v,w\},\{u_1,u_2\}\}),
\end{equation}
obtaining that $P$ is not $L$-maximal.

Therefore, setting $u_0=v$, being $P$ $L$-maximal, by repeating the previous argument we have edges  $$L=\{u_{2i},u_{2i+1}\}\in P \mbox{ for }i=0,\ldots,\ell,$$
 where either $u_{2\ell+1}$ is a free vertex of a whisker, or lies on a cycle.  If $u_{2\ell+1}$ lies on a path or on an even cycle, then we put $\{w,u_1,u_{3}\ldots,u_{2\ell+1}\}$ in $A_P$, that is we remove $\ell+1$ edges in $L$ from $B_P$ and add $\ell+2$ vertices to $A_P$ obtaining a bigger matching. Hence $P$ is not $L$-maximal.  If $u_{2\ell+1}$ lies on an odd cycle,  then we put $\{w,u_1,u_{3}\ldots,u_{2\ell-1}\}$ in $A_P$, that is we remove $\ell$ edges from $B_P$ and add $\ell$ vertices to $A_P$, obtaining an equivalent CI-matching containing $w \in A_P$.

\end{proof}

\begin{lemma}\label{lem:LRmax}
Let $G$ be a connected graph and let $P$ be a $L$-maximal matching of $G$ with $a_P >0$. Then $P$ is not $R$-maximal.
\end{lemma}
\begin{proof}
Let $P=(A, B)$ as in Definition \ref{def:CI}. Since $a_P>0$, then for $v\in A$ we have $N(v)\cap V(B) = \emptyset$. Hence, for $w \in N(v)$ we have that
\[
P'=B \cup \{\{v,w\}\}
\]
is a CI-matching of $G$ such that $b_{P'} \geq b_P$.
\end{proof}
\begin{remark}\label{rmk:alg}
Even though the statement of Lemma \ref{lem:LRmax} is obvious, its proof gives a way to enlarge the clique matching $B$  of a CI-matching $P=(A,B)$ up to a larger clique-matching $B'$. In fact, let $P$ be a $L$-maximal matching, and we take $P_1=B \cup \{\{v,w\}\}$ as in the proof. Then, we take a vertex $v_1 \in A_P \setminus N(w)$ and we have $N(v_1)\cap V(P_1) = \varnothing$, hence for $w_1 \in N(v_1)$ we have 
\[
P_2=P_1 \cup \{\{v_1,w_1\}\}
\]
and we proceed in this way until we run out of vertices in $A_P$ that are not adjacent to a clique in the matching. Hence we build a clique matching $P'=B'$ such that 
\begin{enumerate}
\item $B\subset B'$ by construction;
\item $A'= V(G)\setminus V(B')$ is independent set, because if two vertices are adjacent, then an edge can be added to $B'$;
\end{enumerate}
hence we let $a'=|A'|$ and $b'=|B'|$. Hence, $n=a'+b'$ and $2n-b'=b'+2a'$, hence $2a_P+b_P\leq \cmm \leq 2a'+b'$.
\end{remark}
\begin{proposition}\label{prop:eq}
Let $G$ be a connected graph and let $P$ be a CI-matching of $G$ with $a_P >0$ and let $A',B',a',b'$ as in Remark \ref{rmk:alg}. Then the following are equivalent:
\begin{enumerate}
	\item $A' \subset A_P$;
	\item $2a_P+b_P=2a'+b'$
\end{enumerate}
\end{proposition}
\begin{proof}
	By Remark \ref{rmk:alg} and setting $F=V(G)\setminus V(P)$ with $f=|F|$, we have the following partitions of the vertices of $G$
	\begin{equation}\label{eq:partitions}
	V(G)=A_P\cup V(B_P)\cup F=A'\cup V(B'),
	\end{equation}
	that induce the equation
	\[
	n=a_P+b_P+f=a'+b'.
	\]
	$(1)\Rightarrow (2).$ Since $A' \subset A_P$  there is a bijection between the vertices in $A_P\setminus A'$ and the cliques of cardinality 2, namely the edges of $B'\setminus B_P$, therefore 
	\[
	b'-b_P=2(a_P-a').
	\]
	And the assertion follows.\\
	$(2)\Rightarrow (1).$ Let $s$ be the number of edges $\{v,w\}$ with $v \in A_P$ and $w \in F$ such that $\{v,w\}$ is in $B'$ as in Remark \ref{rmk:alg}.  It follows that $b'=b_{P}+2s$, and
	\[
	 a'=n-b'=(a_P+b_P+f)-(b_{P}+2s)=a_P+f-2s.
	\]
	Since $2a_P+b_P=2a'+b'$, then 
	\[
	2a_P+b_P=2a_P+2f-4s+b_{P}+2s \ \  \Rightarrow \ \  0=2f-2s  \ \ \Rightarrow \ \  f=s.
	\]
	It follows that all the vertices of $F$ are in $B'$ and hence $A' \subset A_P$.
\end{proof}
\begin{corollary}\label{cor:LR}
Let $G$ be a connected graph satisfying one of the two equivalent conditions of Proposition \ref{prop:eq}. Then $G$ has a CI-matching $P$ that is $L$-maximal and $R$-maximal and 
\[
\cmm = 2n - b(G)
\]
\end{corollary}
\begin{remark}
If $G$ is bipartite, according to Theorem \ref{the:bipartite}, we have $\cmm=2d$, where $d$ is the cardinality of a maximum independent set $A$. We observe that $A$ induces an $L$-maximal matching and from K\"onig theorem, $2d=2(n-m(G))=2n-b(G)$, because the maximal matching is a clique matching of $G$ consisting of edges. Hence, the equality of Corollary \ref{cor:LR} holds for bipartite graphs.
\end{remark}
We provide an example of graph that does not admit a CI-matching that is both $L$-maximal and $R$-maximal.
\begin{example}
Let $G$ be the cycle $C_5$ (see Figure \ref{fig:C5}) we have that a maximal CI-matching is $P=(\{1\}, \{\{3,4\}\})$, hence $2a_P+b_P=4$. By applying the algorithm in Remark \ref{rmk:alg}, we obtain $P'=\{\{1,2\},\{3,4\}\}$ and $A'=V(G)\setminus V(P')=\{5\}$, hence $2a'+b'=6$. This happens because $2$ and $5$ are not vertices in the matching $P$, and in $P'$ we can only include one of them. 
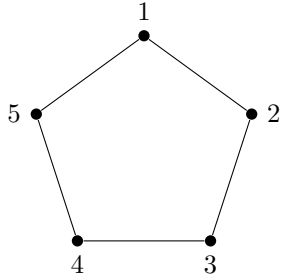
\begin{figure}[h]
\begin{tikzpicture}[scale=1.5]

% Vertici del ciclo C5
\node[circle,fill=black,inner sep=1.5pt,label=above:$1$] (v1) at (90:1) {};
\node[circle,fill=black,inner sep=1.5pt,label=right:$2$] (v2) at (18:1) {};
\node[circle,fill=black,inner sep=1.5pt,label=below:$3$] (v3) at (-54:1) {};
\node[circle,fill=black,inner sep=1.5pt,label=below:$4$] (v4) at (-126:1) {};
\node[circle,fill=black,inner sep=1.5pt,label=left:$5$] (v5) at (162:1) {};

% Archi del ciclo
\draw (v1)--(v2)--(v3)--(v4)--(v5)--(v1);

\end{tikzpicture}
\caption{A cycle $C_5$}\label{fig:C5}
\end{figure}
\end{example}

\section{Block graphs and other families}\label{sec:blocks}
In this section we analyse block graphs, proving that $\cmm = 2n- b(G)$ and other families of graphs such as graphs with 2-factors and cycles with whiskers.
Let $G$ be a graph on the vertex set $V(G)=[n]$ and consider $v\notin V(G)$. The {\em join} of $v$ on $G$, denoted by $v \ast G$, is the graph on the vertex set $V(G)\cup\{v\}$ and edge set $E(G)\cup\{\{v,w\}\ :\ w\in V(G)\}$.

We start by analysing block star graphs. Given a block star $G$ with cutpoint $v$ we write,
\[
G\setminus\{v\}=K_{n_1}\cup \ldots \cup K_{n_r} \cup \{w_1\} \cup \ldots \cup \{w_s\}
\]
where $n_i \geq 2$ for $i=1,\ldots,r$ and $w_j$ for $j=1,\ldots,s$ are isolated vertices.

\begin{proposition}\label{prop:BS}
Let $G$ be a block star in the previous setting. 
Then
\[
\cmm=\begin{cases}
n &\mbox{if } s = 0;\\
n+s-1 &\mbox{if } s \geq 1;
\end{cases}
\]
\end{proposition}
\begin{proof}
Observe that the second formula coincides with the first one for $s=1$. Hence, we prove that for $s=0,1$ we have $\cmm=n$. We observe that $K_{n_1} \cup \ldots \cup K_{n_r}$ is itself a clique matching of the non-isolated vertices of $G\setminus \{v\}$. Hence, if $s=0$, $(v \ast K_{n_1} ) \cup \ldots \cup K_{n_r}$ is a clique matching of $G$, and if $s=1$, $K_{n_1} \cup \ldots \cup K_{n_r} \cup \{v,w_1\}$ is a clique matching of $G$. \\
If $s\geq 1$ we observe that $n=n_1+\ldots + n_r+s+1$ and that $G\setminus \{v\}$ gives rise to a L-maximal CI-matching with $A_P=\{w_1,\ldots,w_s\}$, and hence $a_P=s$, and $b_P=n_{1}+\ldots+n_{r}=n-s-1$. By applying Remark \ref{rmk:alg}, one constructs the clique matching $B'=K_{n_1} \cup \ldots \cup K_{n_r} \cup \{v,w_1\}$ that has an associated $A'=V(G)\setminus B'=\{w_2,\ldots,w_s\} \subseteq A_P$. Therefore, from Proposition \ref{prop:eq} the assertion follows.
\end{proof}

\begin{proposition}\label{pro:algopartblockgraph}
	Let $G$ be a whisker-free connected block graph and let $\mathcal{F}=\{F_1,\ldots,F_r\}$ be the set of maximal cliques of $G$. Then the following algorithm computes a clique matching $\mathcal{P}$ of $G$.
\begin{enumerate}
	\item $\mathcal{P}= \{\}$;
	\item add$\_\PP$($F_1$, $\mathcal{F}\setminus\{F_1\}$);
	\end{enumerate}
where for $F \in \mathcal{F}$ and $\mathcal{F}' \subset \mathcal{F}$ the recursive function of point (2) is defined as follows:

add$\_\PP$($F$, $\mathcal{F}'$): 
	\begin{enumerate}
		\item[(a)] let $F'=V(F)\setminus V(\PP)$;
		\item[(b)] if $|F'|\geq 2$ then $\PP=\PP\cup \{F'\}$;
		\item[(c)] $\forall G\in \FF'$  with $G\cap F'\neq \emptyset$ then add$\_\PP$($G$, $\mathcal{F}\setminus\{G\}$).
	\end{enumerate}
\end{proposition}
\begin{proof}
	Lines $(1)-(2)$. In the beginning we set $\PP$ as the empty set and we call the recursive function add$\_\PP$, choosing $F_1$ as the first clique to add to the matching. We observe that we can start by any maximal clique of the block graph.\\
	Then, we define the parameters of the recursive function add$\_\PP$:
	\begin{itemize}
		\item $F$, namely the clique to analyze and, if needed, from which we can take the subclique to add to the matching;
		\item $\mathcal{F}'$, the set of remaining facets to consider when the function is called. %We consider $\FF\setminus \FF'$ the descendants of $\FF'$.  
	\end{itemize} 
	We explain the lines $(a)$, $(b)$ and $(c)$ that are the the body of the function.
	
	Line $(a)$. Here we define the candidate to add to the matching, namely the set $F'$. We observe that we remove, if needed, the vertex of the clique, namely the vertex that $F$ shares with the parent that is in the set of vertices already in $\PP$. Observe that if $F$ is $F_1$ then $F'$ is exactly $F$, since the matching is empty at the beginning. But it is not the only case as we see in the next Line.
	 
	Line $(b)$. Here, we add $F'$ if and only if the clique under analysis has cardinality greater than $2$, or it has cardinality $2$ but its parent has cardinality $2$ and was skipped for this reason. This is the other case in which $F'$ is exactly $F$.
	
	Line $(c)$. Call the recursive function for all children of $F$.
	We observe that the recursive function ends whenever $F'$ is an end block, namely $F$ has no descendants. Since the block graph is whisker free by hypothesis $F'$ is added in line (b).
\end{proof}
\begin{corollary}
Let $G$ be a block graph with at most one whisker. Then 
\[
\cmm=n.
\]	
\end{corollary}
\begin{proof}
	If $G$ is whisker-free we apply Proposition \ref{pro:algopartblockgraph}, obtaining a clique matching of $G$. If it has a unique whisker with vertices $\{u,v\}$ we apply the Algorithm of the Proposition \ref{pro:algopartblockgraph}, with $F_1=\{u,v\}$, obtaining a clique matching, too.
	In both cases by Proposition \ref{pro:lowupbound} we obtain $b(G)= n\leq \cmm\leq 2n-n$, and the assertion follows. 
\end{proof}
\begin{remark}\label{rem:Lmax}
Given a block graph $G$ that is not clique-partitionable, we can take an L-maximal CI-matching $P=(A_{P},B_P)$, with $A_P \neq \varnothing$ such that: 
\begin{enumerate}
\item any free vertex of the graph is in $V(P)$, in particular if a vertex $v$ is adjacent to free vertices, then $v \notin A_P$;
\item if $v\notin V(P)$, then there exist $w,z \in N(v)$ such that $w \in A_P$ and $z \in V(P)$. In particular, if $N(v) \cap A_P = \{w\}$, then we take the CI-matching $P'=(A_{P'}, B_{P'})$ with $A_{P'}=A_{P}\setminus \{w\}$ and $B_{P'}=B_{P}\cup \{\{v,w\}\}$.
\end{enumerate}
\end{remark}
\begin{proposition}\label{pro:ci_bG}
Let $G$ be a block graph and let $P$ be an L-maximal CI-matching of $G$. Then, there exists a clique matching $Q$ of $G$ of cardinality $b(G)$, such that
\begin{itemize}
\item $V(B_P) \subseteq V(Q)$;
\item $V(G) \setminus V(Q) \subseteq A_P$.
\end{itemize}
\end{proposition}
\begin{proof}
We proceed by induction on the number $r$ of cutpoints of $G$. \\ If $r=0$, $G$ is a complete graph, namely $P=Q$ is a clique cover,  and the assertion is true. \\ If $r=1$, the assertion follows from the proof of Proposition \ref{prop:BS}.  \\
If $r>1$, then we take a cutpoint $v$ that disconnects only one non-complete component. We recall that the  L-maximal CI-matching  is $P=(A_P,B_P)$. If $A_P=\varnothing$, then the assertion follows, since $P$ is a clique matching itself. Let $A_P \neq \varnothing$. By Remark \ref{rem:Lmax}.(1) we have $v \notin A_P$ and we distinguish three cases up to equivalence of CI-matchings: 
\begin{enumerate}
\item $|N(v) \cap A_P| \geq 2$;
\item $|N(v) \cap A_P| = 1$;
\item $|N(v) \cap A_P| = 0$.
\end{enumerate}

In case (1), we have that $v \notin B_P$ by construction, and hence $P$ is a CI-matching of $G\setminus v$. Since $G\setminus v$ is a block graph with $r-1$ cutpoints, by inductive hypothesis there exists $Q'$ maximal clique matching of $G\setminus v$ such that  $V(B_P) \subseteq V(Q')$, 
$V(G) \setminus V(Q') \subseteq A_P$. Since $|N(v) \cap A_P| \geq 2$ and  $v$ disconnects only one non-complete component, then there is a whisker $\{v,w\} \in E(G)$, hence 
$Q=Q' \cup \{\{v,w\}\}$ is a maximal clique matching of $G$ that satisfies the desired properties. \\
In case (2),  if $w$ is the unique element of $ N(v) \cap A_P$, then from Remark \ref{rem:Lmax}.(2) we can take $P$ such that $\{v,w\} \in B_P$.
We observe that $P'= (A_{P}, B_{P}\setminus \{v,w\})$ is a maximal CI-matching of $G\setminus \{v,w\}$. From the inductive hypothesis there exists $Q'$ maximal clique matching of $G\setminus \{v,w\}$ such that  $V(B_{P}) \subseteq V(Q')$ and $V(G) \setminus V(Q') \subseteq A_{P}$. Therefore $Q=Q' \cup \{v,w\}$ is a maximal clique matching of $G$ that satisfies the desired properties.\\
In case (3), observe that $v \in V(Q) \subset V(B_P)$ for some clique $Q$ and $|V(Q)| \geq 3$. In fact, if $\{v,w\} \in B_P$ for some $w$, then the CI-matching $P'=(A_P\cup \{w\},B_P \setminus \{v,w\})$ satisfies condition (2).
We take $H=K_m$ such that $v \in V(H)$ and we take $\hat{H}=H \setminus v$. We observe that $P'= (A_P,B_P\setminus H \cup \hat{H})$ is a maximal CI-matching of $G\setminus v$. From the inductive hypothesis there exists $Q'$ maximal clique matching of $G\setminus v$ such that  $V(B_P) \subseteq V(Q')$
$V(G) \setminus V(Q') \subseteq A_P$. Therefore $Q=Q'\setminus \hat{H} \cup H$ is a maximal clique matching of $G$ that satisfies the desired properties. \\

\end{proof}
\begin{theorem}\label{the:blockgraph}
Let $G$ be a block graph. Then 
\[
\cmm=2n-b(G).
\]	
\end{theorem}
\begin{proof}
Follows from Proposition \ref{pro:ci_bG} and \ref{prop:eq}.
\end{proof}
\subsection{Graphs with a 2-factor}
In this section, we recover $c_G(\mm)$ for graphs having a 2-factor.

A graph has a $2$-factor if there are disjoint cycles spanning $G$. We remark that this class of graphs contains the Hamiltonian graphs.
\begin{theorem}\label{thm:2fac}
	Let $G$ be a graph on $n$ vertices that has a $2$-factor. Then
	\[
	\cmm=n.
	\]
\end{theorem}
\begin{proof}
Let $H$ be the disjoint union of cycles $C_{i_1},\ldots, C_{i_r}$ spanning $G$. By Theorem \ref{the:oddcycles} we have $c_{C_{i_j}}(\mm)=i_j$ for $j \in 1,\ldots , r$. 
Hence, from Lemma \ref{lem:tp} we have $c_{H}(\mm)=\sum_{j=1}^r i_j =n$, and from  Lemma \ref{lem:subgraph} we have that 
		\[
		\cmm\leq c_H(\mm')=n 
		\]
	 The other inequality is induced by Lemma \ref{lem:geq}.
\end{proof}	

\begin{example}
Let $G$ be the Petersen graph, see Figure \ref{fig:pet}, that is well-known having a $2$-factor. From Theorem \ref{thm:2fac}, we obtain that 
\[
\cmm=n.
\]
We moreover observe that it contains a perfect matching $\{\{i,i+5\} : i \in \{1,\ldots 5\} \}$, hence we get $\cmm=n$ also from  Proposition \ref{prop:pm}.
\begin{figure}[H]
\begin{tikzpicture}[scale=2]

% Outer pentagon vertices (1–5)
\node[circle,fill=black,inner sep=1.5pt] (v1) at (90:1) {};
\node[circle,fill=black,inner sep=1.5pt] (v2) at (18:1) {};
\node[circle,fill=black,inner sep=1.5pt] (v3) at (-54:1) {};
\node[circle,fill=black,inner sep=1.5pt] (v4) at (-126:1) {};
\node[circle,fill=black,inner sep=1.5pt] (v5) at (162:1) {};

% Inner star vertices (6–10)
\node[circle,fill=black,inner sep=1.5pt] (u1) at (90:0.45) {};
\node[circle,fill=black,inner sep=1.5pt] (u2) at (18:0.45) {};
\node[circle,fill=black,inner sep=1.5pt] (u3) at (-54:0.45) {};
\node[circle,fill=black,inner sep=1.5pt] (u4) at (-126:0.45) {};
\node[circle,fill=black,inner sep=1.5pt] (u5) at (162:0.45) {};

% Outer cycle
\draw (v1)--(v2)--(v3)--(v4)--(v5)--(v1);

% Spokes
\draw (v1)--(u1);
\draw (v2)--(u2);
\draw (v3)--(u3);
\draw (v4)--(u4);
\draw (v5)--(u5);

% Inner star
\draw (u1)--(u3)--(u5)--(u2)--(u4)--(u1);

\end{tikzpicture}
\caption{Petersen graph}\label{fig:pet}
\end{figure}
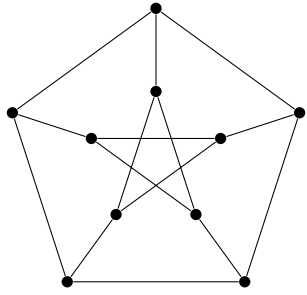

\end{example}

\subsection{Cycles with whiskers}
In this section, we compute our invariant for cycles with whiskers.  We start by the following notation. Given $w_1,\ldots, w_r \in \{1,\ldots,n\}$ we set
\[
S_{w_1w_2\ldots w_r}=K[{x_i,y_i \ : \ i \in \{1,\ldots,n\} \setminus \{w_1,\ldots,w_r\} }]
\].

\begin{lemma}\label{lem:whi}
	Let $G$ be a graph and let $e=\{u,v\}$ be a whisker of $G$. Then 
	\[
	\cmm = c_{G\setminus e}(\mm') +2
	\]
	where $\mm'$ is the maximal ideal of $S_{uv}$.
\end{lemma}
\begin{proof}
	The disconnected graph $(G\setminus e)\cup E$, with $E$ the graph containing the edge $e$, is a subgraph of $G$, hence from Lemma \ref{lem:subgraph} and Theorem \ref{the:bipartite}  we have $$\cmm \leq c_{G\setminus e}(\mm') +c_E(\mm'')=c_{G\setminus e}(\mm') +2$$ with $\mm''\subset K[x_u,y_u,x_v,y_v]$.
	Conversely, let $k \in \NN$ and let $\lambda$ be a monomial of degree $v_{G\setminus e}(k)$, and $\lambda\notin (\mm'^{[k]},J_{G\setminus e})$.
	
	Assume  $u$ the free vertex of $G$ in $e$, using a similar argument of Lemma \ref{lem:independentset}, we obtain $(x_{u}y_{u})^{k-1} \lambda \notin (\mm^{k},J_G)$ , and hence 
	\[
	v_{G}(k) \geq v_{G\setminus e}(k)+2(k-1).
	\]
	Taking the limit the desired inequality follows.
\end{proof}

If instead of a whisker one considers a complete graphs $K_m$, then we do not obtain an exact formula.
\begin{proposition}\label{lem:cli} 
	Let $G$ be a graph, let $K_m$ be a complete graph on vertices $\{v_1,\ldots,v_m\}$ with only one non-free vertex $v_m$ and $m\geq3$, and let $c=c_{G\setminus K_m}(\mm')$ where $\mm'$ is the maximal ideal of $S_{v_1v_2\ldots v_m}$.Then 
	\[
	\cmm \in \{ c+m-1,c+m\}.
	\]
\end{proposition}
\begin{proof}
	We prove 
	\[
	c+m-1\leq \cmm \leq c+m
	\]
	The disconnected graph $G\setminus K_m \cup K_m$ is a subgraph of $G$, hence from Lemma \ref{lem:subgraph} we have $\cmm \leq c +m$.
	Conversely, let $k \in \NN$ and let $\lambda$ be a monomial of degree $v_{G\setminus K_m}(k)$. Then, $(x_{v_1}\cdots x_{v_{m-1}})^{k-1} \lambda \notin (\mm^{k},J_G)$, and hence 
	\[
	v_{G}(k) \geq v_{G\setminus K_m}(k)+(m-1)(k-1)
	\]
	and taking the limit the desired inequality follows.
\end{proof}

\begin{theorem}\label{thm:2cycleswithwhiskers}
	Let $G$ be a cycle with at least one whisker. Then
	\[
	\cmm=2n - b(G).
	\]
\end{theorem}
\begin{proof}
	
	Let $e$ be a whisker of $G$ and $H=G\setminus e$. By Lemma \ref{lem:whi} we have
	\[
	\cmm = c_{H}(\mm') +2.
	\]
	Since $H$ is a is a block graph, in particular a tree, by Theorem \ref{the:blockgraph} we have
	$c_{H}(\mm')=2(n-2)-b(H)$, namely
	\[
	\cmm = 2(n-2)-b(H)+2=2(n-2)+b(H)+4-2=2n-(b(H)+2).
	\]
	Now, we prove that $b(G)=b(H)+2$. In fact, let $P$ be a clique matching of $H$ of cardinality $b(H)$, then adding to $P$ the edge $e$ we obtain a clique matching $P'$ of $G$ of cardinality $b(H)+2$.
	Moreover, assuming that there exists a clique matching $Q$ of cardinality bigger than $b(H)+2$, we claim that there exists one of the same cardinality containing the edge $e$. The partition containing $e$  induces a partition $Q'$ on $H$, obtained removing  the edge $e$ from $Q$, of cardinality greater than $b(H)$ that is absurd. 
	If $e \in Q$ there is nothing to prove. Otherwise,  let $e=\{v,w\}$ with $w$ isolated. There exists $z \in N(v)$ such that $\{v,z\} \in Q$. We take $Q'=Q\setminus \{\{v,z\}\} \cup \{e\}$. The claim follows.
	   
\end{proof}

\end{document}